\documentclass[12pt,reqno]{article}
\usepackage{amssymb,amscd,amsmath,amsthm,color}
\newtheorem{theorem}{Theorem}[section]
\newtheorem{lemma}[theorem]{Lemma}
\newtheorem{cor}[theorem]{Corollary}
\newtheorem{prop}[theorem]{Proposition}

\newtheorem{definition}[theorem]{Definition}
\usepackage{enumerate,amssymb}
\theoremstyle{definition}
\newtheorem{example}[theorem]{Example}

\theoremstyle{remark}
\newtheorem{remark}[theorem]{Remark}

\numberwithin{equation}{section}

\def\cM{\mathcal{M}}
\def\b1{\mathrm{1}}

\def\bC{\mathbb{C}}

\def\bM{\mathbb{M}}
\def\bN{\mathbb{N}}
\def\bR{\mathbb{R}}

\def\ffi{\mathcal{\varphi}}

\def\cM{\mathcal{M}}
\def\cN{\mathcal{N}}

\def\eps{\varepsilon}
\def\cS{\mathcal{S}}
\def\cH{\mathcal{H}}

\begin{document}
\baselineskip=15pt

\title{Anti-norms on finite von Neumann algebras}

\date{ }

\author{Jean-Christophe Bourin{\footnote{Supported by ANR 2011-BS01-008-01.}}
\ and
Fumio Hiai{\footnote{Supported by Grant-in-Aid for Scientific Research (C)26400103.}}}

\maketitle

\begin{abstract}
\noindent
As the reversed version of usual symmetric norms, we introduce the notion of symmetric
anti-norms $\|\cdot\|_!$ defined on the positive operators affiliated with a finite von
Neumann algebra with a finite normal trace. Related to symmetric anti-norms, we develop
majorization theory and superadditivity inequalities of the form
$\|\psi(A+B)\|_!\ge\|\psi(A)\|_!+\|\psi(B)\|_!$ for a wide class of functions $\psi$.
\end{abstract}

{\small\noindent

\medskip\noindent
AMS subjects classification 2010: 47A30, 46L52, 46L51}

\section{Introduction}

In Functional Analysis, symmetrically normed Banach functions spaces are classical objects as well as their non-commutative generalizations in the setting of $\tau$-measurable operators affiliated with a von Neumann algebra with a faithful normal finite trace $\tau$. Symmetric norms are homogeneous convex functional completely determined by their values on the positive cone of the function space or the operator algebra. This point of view may motivate the study of concave, homogeneous functionals on positive operators. It is our concern in this article. 

A part of our work could fit in a very general setting, for instance, in the $C^*$-algebra framework. However, we confine to finite von Neumann algebras for two reasons. First, dealing with a von Neumann algebra equipped with a normal finite trace allows to consider unbounded operators, and hence to develop a theory parallel to most of the usual non-commutative Banach function spaces. Secondly, the finiteness assumption allows to consider many functionals which would not make sense in the non-finite case, such as the Fuglede-Kadison determinant.
The assumption is also essential for some technical reasons (for instance, a unitary operator can be taken as a phase of the polar decomposition).
Moreover, this setup of a finite von Neumann algebra naturally extends the matrix approach in our previous work. 

We call our functionals, defined on the positive part $\cN^+$ of a finite von Neumann algebra $\cN$, {\it symmetric anti-norms} on $\cN^+$, as the triangle inequality for norms is then reversed. Section 2 gives the precise definition and exhibits an important family of such anti-norms which are derived from symmetric norms. For the convenience of the reader and to fix terminologies which may have some variants in the literature, our discussion also covers basic facts on symmetric norms. Our approach of symmetric norms may have its own interest and originality. Section 4 is devoted to the (non-obvious) extension of these anti-norms to the whole set of densely-defined positive operators affiliated with $\cN$. Here we consider the more classical case of symmetric norms as well. Section 5 presents a superadditivity inequality for convex functions which is a far extension of a classical trace inequality of Rotfel'd. Several norm inequalities follow from this anti-norm inequality. In Sections 6 and 7, we focus on a special class of symmetric norms and anti-norms which corresponds, in the commutative case, to the class of rearrangement invariant function spaces. The theory is then related to majorization relations.

Most of the results, norm and anti-norm inequalities given in Sections 4--7, are based on operator inequalities via unitary orbits. These essential operator inequalities are established in Section 3. The idea of the proofs consists in combining a unitary orbit technique for spectral dominance in a finite factor with the disintegration of $\cN$ into its factorial components. These results nicely extend the scope of some well-known matrix inequalities to the general finite von Neumann algebra setting.

\section{Symmetric norms and symmetric anti-norms}

Let $\cN$ be a finite von Neumann algebra acting on a separable Hilbert space $\cH$ with a
faithful normal finite trace $\tau$, and $\cN^+$  the set of positive operators in $\cN$.
Let $\overline\cN$ denote the set of $\tau$-measurable operators affiliated with $\cN$ (see
\cite{FK} for details), and $\overline\cN^+$ the positive cone of $\overline\cN$. Since
$\tau$ is finite, $\overline\cN$ is  the set of all densely-defined closed operators
affiliated with $\cN$.

In this article, a {\it symmetric norm} $\|\cdot\|$ on $\cN$ means a norm satisfying
$\|UXV\|=\|X\|$ for all $X\in\cN$ and all unitaries $U,V\in\cN$. The monotonicity of such
a norm in the next lemma is a well-known simple fact \cite[Lemma 3.2, Corollary 3.3]{FHNS}.
We give an alternative proof. The letter $I$ stands for the identity (of any algebra).

\begin{lemma}\label{monot}
Any symmetric norm $\|\cdot\|$ on $\cN$  is monotone, i.e., $\|A\|\le\|B\|$ if $A,B\in\cN^+$ and $A\le B$.
\end{lemma}

\begin{proof}
Let $T\in\cN$ be a contraction. Note that $|T|=(V_1+V_2)/2$ where
$V_1:=|T|+i\sqrt{I-|T|^2}$ and $V_2:=|T|-i\sqrt{I-|T|^2}$ are unitaries. As $T=U|T|$ for some
unitary $U\in\cN$ (since $\cN$ is a finite von Neumann algebra), we have $T=(U_1+U_2)/2$ with two unitaries $U_1,U_2\in\cN$. Therefore
$\|TXT\|\le\|X\|$ for all $X\in\cN$. Now, assume that $0\le A\le B$ in $\cN$. Then there
exists a contraction $C\in\cN^+$ such that $A=B^{1/2}CB^{1/2}=VC^{1/2}BC^{1/2}V^*$
for some unitary $V\in\cN$. Hence $\|A\|=\|C^{1/2}BC^{1/2}\|\le\|B\|$.
\end{proof}

Consequently, any symmetric norm on $\cN$ is continuous with respect to the operator norm
$\|\cdot\|_\infty$. In fact, since $|X|\le\|X\|_\infty I$,  we have
\begin{equation}\label{norm-ineq}
\|X\|\le\|X\|_\infty\|I\|,\qquad X\in\cN.
\end{equation}

A symmetric norm $\|\cdot\|$ on $\cN$ depends only on its values on positive operators via
the polar decomposition, and its restriction to $\cN^+$ satisfies
\begin{itemize}
\item[(1)] $\|\alpha A\|=\alpha\|A\|$ for all $A\in\cN^+$ and all scalars $\alpha\ge0$,
\item[(2)] $\|A\|=\|UAU^*\|$ for all $A\in\cN^+ $ and all unitaries $U\in\cN$,
\item[(3)] $\|A\|\le\|A+B\|\le\|A\|+\|B\|$ for all $A,\,B\in\cN^+ $.
\end{itemize}
The first inequality in (3) follows from Lemma \ref{monot}. Conversely, if $\|\cdot\|$ is
a non-negative functional on $\cN^+$ satisfying (1)--(3), then $\|X\|:=\|\,|X|\,\|$ for
$X\in\cN$ becomes a symmetric norm (more precisely, semi-norm) on $\cN$, as immediately
shown by a triangle inequality in \cite{AAP} or by Proposition \ref{prop-convex} below.

We introduce the notion of symmetric anti-norms on the positive cone $\cN^+$,
by replacing the convexity/subadditivity of symmetric norms with concavity/superadditivity.

%

\begin{definition}\label{Def-1}\rm
A {\it symmetric anti-norm} $\|\cdot\|_!$ on $\cN^+$ is a functional taking values in
$[0,\infty)$
satisfying the following properties:
\begin{itemize}
\item[$(1)_!$] $\|\alpha A\|_!=\alpha\|A\|_!$ for all $A\in \cN^+$ and all scalars
$\alpha\ge 0$,
\item[$(2)_!$] $\|A\|_!=\|UAU^*\|_!$ for all $A\in\cN^+ $ and all unitaries $U\in\cN$,
\item[$(3)_!$] $\|A+B\|_!\ge\|A\|_!+\|B\|_!$ for all $A,\,B\in\cN^+$.
\item[$(4)_!$] $\|A+\eps I\|_!\searrow\|A\|_!$ as $\eps\searrow0$ for all $A\in\cN^+$.
\end{itemize}
\end{definition}

This definition was first introduced in \cite{BH1,BH2} for the matrix algebra $\bM_n$.
In the matrix case, $(4)_!$ is equivalent to the usual continuity with respect to the
operator norm.

Typical example of symmetric anti-norms are $A\mapsto \{\tau(A^q)\}^{1/q}$, $0< q\le 1$, and $A\mapsto \{\tau( A^{-1/p})\}^{-1/p}$, $0<p<\infty$. The latter is first defined on the invertible part of $\cN^+$, and understood for non-invertible operators as
$$
\{\tau( A^{-p})\}^{-1/p}:=\lim_{\eps\searrow 0} \{\tau(( A+\eps I)^{-p})\}^{-1/p},
$$
where the finiteness assumption $\tau(I)<\infty$ is essential to have non-trivial functionals on $\cN^+$. These Schatten like functionals with negative exponents are a special case of a more general family.

\begin{definition}\label{Def-2}\rm
Fix a symmetric norm $\|\cdot\|$ on $\cN$ and $p>0$. For each $A\in\cN^+$, since
$\|(A+\eps I)^{-p}\|^{-1/p}$ decreases as $\eps\searrow0$ by Lemma \ref{monot}, we can
define
$$
\|A\|_!:=\lim_{\eps\searrow0}\|(A+\eps I)^{-p}\|^{-1/p}.
$$
Note that if $A$ is invertible, then the above $\|A\|_!$ is equal to $\|A^{-p}\|^{-1/p}$,
i.e.,
\begin{equation}\label{form-deriv}
\|A^{-p}\|^{-1/p}=\lim_{\eps\searrow 0}\|(A+\eps I)^{-p}\|^{-1/p}.
\end{equation}
We call this functional $A\in\cN^+\mapsto\|A\|_!$ a {\it derived anti-norm} and say that
it is derived from $\|\cdot\|$ and $p$.
\end{definition}

A derived anti-norm is indeed a symmetric anti-norm as claimed in the next statement.

\begin{theorem}\label{theorem-deriv}
The above functional $\|\cdot\|_!$  derived from a symmetric norm $\|\cdot\|$ on $\cN$ and
a $p>0$ satisfies 
$$
\| A+B\|_! \ge \| A\|_! + \| B\|_!
$$
for every $A,B\in\cN^+$. Hence $\|\cdot\|_!$ is a symmetric anti-norm on $\cN^+$.
\end{theorem}

To prove this result, we begin with an operator arithmetic-geometric mean inequality.

\begin{lemma}\label{AGM}
Let $A,B\in\cN^+$. Then, there exists a unitary $V\in \cN$ such that
$$
|BA| \le \frac{A^2 +VB^2V^*}{2}.
$$
\end{lemma}

\begin{proof} Consider the operator on ${\mathcal{H}}\oplus {\mathcal{H}}$, 
$$
\begin{bmatrix}
A^2 &AB \\
BA &B^2
\end{bmatrix}
$$
which is a positive operator. Thus, for any $V\in\cN$ so is 
$$
\begin{bmatrix}
I & -V
\end{bmatrix}
\begin{bmatrix}
A^2 &AB \\
BA &B^2
\end{bmatrix}
\begin{bmatrix}
I \\ -V^*
\end{bmatrix}.
$$
Letting $V^*$ be the unitary factor in the polar decomposition $BA=V^*|BA|$ yields the inequality.
\end{proof}

\vskip 5pt
Combining Lemmas \ref{monot} and \ref{AGM} yields $\| AB \| \le (s\| A^2\| + s^{-1}\| B^2\|)/2$ for all symmetric norms, $A,B\in\cN^+$, and $s>0$. Thus, minimizing over $s$ and considering  operators $X,Y\in\cN$ with $|X^*|=A$, $|Y^*|=B$, we obtain the Cauchy-Schwarz inequality for symmetric norms.

\begin{cor}\label{Schwarz} Let $X,Y\in\cN$. Then, for any symmetric norm on $\cN$, 
$$
\| X^*Y \| \le \| X^*X \|^{1/2} \| Y^*Y\|^{1/2}.
$$
\end{cor}

As a byproduct of this  inequality we get from $\|\cdot\|$ another symmetric norm.

\begin{cor}\label{Qnorm} If $\|\cdot\|$ is a symmetric norm on $\cN$, then
$X\mapsto \| X^*X\|^{1/2}$ is a symmetric norm on $\cN$ too.
\end{cor}

For symmetric anti-norms, the following proposition is known in the matrix case \cite{BH1}.

\begin{lemma}\label{lem-aux1} If $\|\cdot\|_!$ is a symmetric anti-norm on $\cN^+$ and
$0<q<1$,  then $A\mapsto \| A^q\|_!^{1/q}$ is a symmetric anti-norm on $\cN^+$ too.
\end{lemma}

\begin{proof} The same proof as in the matrix case \cite{BH1} shows that it is a
homogeneous, unitarily invariant and concave functional. The continuity property $(4)_!$,
i.e., $\| A^q\|_!^{1/q}=\lim_{\eps\searrow 0} \| (A+\eps)^q\|_!^{1/q}$ is obvious from the
monotonicity of $\|\cdot\|_!$ since $A^q\le(A+\eps I)^q\le A^q+\eps^qI$.
\end{proof}

We are now in a position to prove the theorem.

\begin{proof} (Theorem \ref{theorem-deriv})\quad Let $\|\cdot\|$ be a symmetric norm on
$\cN$. Let $A,B\in\cN^+$ be invertible and assume that $\| A^{-1}\|=\|B^{-1}\|=1$. As
$t\mapsto t^{-1}$ is  operator convex on $(0,\infty)$, we have, for $0<s<1$,
$$
\| (sA+(1-s)B)^{-1} \| \le \|sA^{-1} +(1-s)B^{-1}\| \le s +(1-s) =1
$$
and so 
$$
\| (sA+(1-s)B)^{-1} \|^{-1} \ge 1.
$$
For general invertible $S,T\in\cN^+$, taking in this estimate $A=\| S^{-1}\| S$, $B=\| T^{-1}\| T$, and $s= \| S^{-1}\|^{-1}/ ( \| S^{-1}\|^{-1}+  \| T^{-1}\|^{-1})$ yields
$$
\| (S+T)^{-1} \|^{-1} \ge  \| S^{-1}\|^{-1} +  \| T^{-1}\|^{-1}.
$$
Therefore, $A\mapsto \|A\|_!:=\|A^{-1}\|^{-1}$ is a homogeneous and concave/superadditive  functional on the invertible part of $\cN^+$. It can be extended with the same properties to the whole of $\cN^+$ by the limit formula $\|A\|_!:=\lim_{\eps\searrow 0}\|A+\eps I\|_!$. Hence, this functional derived from a symmetric norm $\|\cdot\|$ and $p=1$ is a symmetric anti-norm on $\cN^+$.

Next we consider a functional derived from a symmetric norm $\|\cdot\|$ and an arbitrary $p>0$. We have $p=2^nq$ where $n$ is a positive integer and $0<q<1$. By Corollary \ref{Qnorm} applied $n$ times and the first step of the proof,
$$
\|A\|_!^{(n)}:=\lim_{\eps\searrow0}\|(A+\eps I)^{-2^n}\|^{-1/2^n}
$$
is a symmetric anti-norm on $\cN^+$. Applying Lemma \ref{lem-aux1} shows that
$A\mapsto\bigl(\|A^q\|_!^{(n)}\bigr)^{1/q}$ is a symmetric anti-norm too, which is readily
verified to be the functional derived from $\|\cdot\|$ and $p$.
\end{proof}

\section{Inequalities via unitary orbits}

This section is a main technical part of this article. The next theorems give superadditive
or subadditive operator inequalities via unitary orbits for convex or concave functions,
which will be of essential use in Sections 4 and 5.

\begin{theorem}\label{th-super}
Let $g(t)$ be a non-negative convex function on $[0,\infty)$ with $g(0)=0$. Then, for every
$A,B\in\overline{\cN}^+$ and every $\eps>0$, there exist unitaries $U,V\in\cN$ such that
$$
g(A+B)+\eps I\ge Ug(A)U^*+Vg(B)V^*.
$$
\end{theorem}

\begin{theorem}\label{th-sub}
Let $f(t)$ be a non-negative concave function on $[0,\infty)$. Then,
for every $A,B\in\overline{\cN}^+$ and every $\eps>0$, there exist unitaries $U,V\in\cN$
such that
$$
f(A+B)\le Uf(A)U^*+Vf(B)V^*+\eps I.
$$
\end{theorem}

Before proving the theorems we recall the notion of the spectral scale \cite{Pe}. The
{\it spectral scale} of $A\in\overline{\cN}^{+}$ is defined as
\begin{equation}\label{sp-scale}
\lambda_t(A):=\inf\{s\in\bR:\tau({\mathbf{1}}_{(s,\infty)}(A))\le t\},
\qquad t\in(0,\tau(I)),
\end{equation}
where ${\mathbf{1}}_{(s,\infty)}(A)$ is the spectral projection of $A$ corresponding to
$(s,\infty)$. We write $\lambda(A)$ for the function $t\mapsto\lambda_t(A)$ on $(0,\tau(I))$,
which is non-increasing and right-continuous. Furthermore, we write $\lambda_0(A)$ and
$\lambda_{\tau(I)}(A)$ for $\lim_{t\searrow0}\lambda_t(A)$ and
$\lim_{t\nearrow\tau(I)}\lambda_t(A)$, respectively, which are the maximal and minimal
spectra of $A$ (when $A$ is bounded). The {\it generalized $s$-numbers} \cite{FK} of $X\in\overline{\cN}$ is
$\mu_t(X):=\lambda_t(|X|)$, $t\in(0,\tau(I))$.

What we will use to prove the theorems is the following lemma. The lemma is rather
well-known but we give the proof for the convenience of the reader.

\begin{lemma}\label{lemma-sp-dominance}
Let $\cN$ be a finite factor and $A,B\in\overline{\cN}^+$. If $B$ spectrally dominates $A$,
i.e., $\lambda_t(A)\le\lambda_t(B)$ for all $t\in(0,\tau(I))$, then for every $\eps>0$ there
exists a unitary $U\in\cN$ such that $UAU^*\le B+\eps I$.
\end{lemma}

\begin{proof}
Since the matrix case is obvious without $\eps I$ in the right-hand side, we may assume
that $\cN$ is a type II$_1$ factor with the normalized trace $\tau$. Choose an increasing
family $\{F_t\}_{0\le t\le1}$ of projections in $\cN$ such that $\tau(F_t)=t$ for all
$t\in[0,1]$. Define
$$
\tilde A:=\int_0^1\lambda_t(A)\,dF_t,\qquad
\tilde B:=\int_0^1\lambda_t(B)\,dF_t.
$$
We then have $\lambda(\tilde A)=\lambda(A)$, $\lambda(\tilde B)=\lambda(B)$ and
$\tilde A\le\tilde B$. Hence the assertion follows since
$\|A-V\tilde AV^*\|_\infty<\eps/2$ and $\|B-W\tilde BW^*\|_\infty<\eps/2$ for some
unitaries $V,W\in\cN$ by \cite[Lemma 4.1]{HN2}.
\end{proof}

We now turn to the proofs of the theorems, which are based on the spectral dominance
theorem \cite{BK} and the central direct decomposition.

\begin{proof} (Theorem \ref{th-super})\quad
First assume that $\cN$ is a finite factor. Then the matrix case is \cite[Theorem 2.1]{AB}
(without $\eps I$ in the left-hand side). The proof in the type II$_1$ factor case is similar
based on \cite{BK}. For any contraction $Z\in\cN$ and any $T\in\overline{\cN}^+$ it is known
\cite[Lemma 10\,(ii)]{BK} that $Z^*g(T)Z$ spectrally dominates $g(Z^*TZ)$. Hence, by
Lemma \ref{lemma-sp-dominance},
\begin{equation}\label{fund-ineq}
Z^*g(T)Z+\eps I\ge Wg(Z^*TZ)W^*
\end{equation}
for some unitary $W\in\cN$. Then, by arguing as in the proof of \cite[Theorem 2.1]{AB} or
\cite[Corollary 3.2]{BL}, one can see that the claimed inequality holds for some unitaries
$U,V\in\cN$.

For the non-factor case, as in \cite{Hi}, we take the central direct integral decomposition
into factors (see \cite{Ta}) as
\begin{equation}\label{decomp-1}
\{\cN,\cH\}=\int_\Omega^\oplus\{\cN_\omega,\cH_\omega\}\,d\nu(\omega),
\qquad\tau=\int_\Omega^\oplus\tau_\omega\,d\nu(\omega)
\end{equation}
over a finite measure space $(\Omega,\mathcal{B},\nu)$ that may be assumed to be complete.
Then $A,B$ are represented as
\begin{equation}\label{decomp-2}
A=\int_\Omega^\oplus A_\omega\,d\nu(\omega),\qquad
B=\int_\Omega^\oplus B_\omega\,d\nu(\omega)
\end{equation}
with unique (a.e.) measurable fields $\omega\mapsto A_\omega,B_\omega$ of 
$\tau_\omega$-measurable positive operators affiliated with $\cN_\omega$. For each
$\omega\in\Omega$, from the first step of the proof, there are unitaries
$U,V\in\cN_\omega$ such that
\begin{equation}\label{cond-1}
g(A_\omega+B_\omega)+\eps I_\omega\ge Ug(A_\omega)U^*+Vg(B_\omega)V^*.
\end{equation}
Now, define $F(\omega)$ to be the set of pairs $(U,V)$ of unitaries in $\cN_\omega$
satisfying \eqref{cond-1}, and prove that there are measurable fields
$\omega\mapsto U_\omega$ and $\omega\mapsto V_\omega$ such that
$(U_\omega,V_\omega)\in F(\omega)$ for all $\omega\in\Omega$. For this, as in \cite{Hi},
we may assume that $\omega\mapsto\cH_\omega$ is a constant field $\cH_0$. Then $F(\cdot)$
is a multifunction whose values are non-empty closed subsets of a Polish space
$B(\cH_0)_1\times B(\cH_0)_1$, where $B(\cH_0)_1$ is the closed unit ball of $B(\cH_0)$
with the strong* topology. By using \cite[Theorem 6.1]{Him} and \cite[Lemma 3.2\,(1)]{Hi}
as well as the fact that $\overline\cN$ with the $\tau$-measure topology is a Polish space, 
we infer that the graph
$$
\{(\omega,U,V)\in\Omega\times B(\cH_0)_1\times B(\cH_0)_1:(U,V)\in F(\omega)\}
$$
of $F(\cdot)$ belongs to $\mathcal{B}\otimes\mathcal{B}(B(\cH_0)_1\times B(\cH_0)_1)$,
where $\mathcal{B}(B(\cH_0)_1\times B(\cH_0)_1)$ is the Borel $\sigma$-field of
$B(\cH_0)_1\times B(\cH_0)_1$. Hence, as in \cite{Hi} the measurable selection theorem
(e.g., \cite{Him}) yields measurable fields
$\omega\mapsto U_\omega$ and $\omega\mapsto V_\omega$ as desired, so we obtain the claimed
inequality with the unitaries $U=\int_\Omega^\oplus U_\omega\,d\nu(\omega)$ and
$V=\int_\Omega^\oplus V_\omega\,d\nu(\omega)$ in $\cN$.
\end{proof}

\begin{proof} (Theorem \ref{th-sub})\quad
The matrix case is in \cite[Theorem 2.1]{AB}. In the type II$_1$ factor case, inequality
\eqref{fund-ineq} for a contraction $Z\in\cN$ and an $T\in\overline{\cN}^+$ is, in turn,
reversed as
$$
Z^*f(T)Z\le Wf(Z^*TZ)W^*+\eps I
$$
by \cite[Lemma 10\,(i)]{BK} and Lemma \ref{lemma-sp-dominance} similarly. (Here, note that
although our assumption on $f$ is slightly weaker than that in \cite{BK}, the proof of
\cite[Lemma 10\,(i)]{BK} can easily be modified to show that $f(Z^*TZ)$ spectrally
dominates $Z^*f(T)Z$.) Hence the desired assertion follows in the factor case. Now, the
proof for the non-factor case is the same as above.
\end{proof}

An idea of the above proofs is to combine a unitary orbit technique with the measurable
selection theorem. We end the section with another illustration of the idea, along the
lines of \cite[Proposition 2.11]{BL} for the matrix case.

\begin{prop}\label{prop-convex}
Let $g(t)$ be a non-decreasing convex function on $[0,\infty)$. Then, for every
$X,Y\in\overline{\cN}$ and every $\eps>0$, there exist unitaries $U,V\in\cN$ such that
$$
g(|X+Y|) \le {Ug(|X|+|Y|)U^*+Vg(|X^*|+|Y^*|)V^*\over2}+\eps I.
$$
\end{prop}

\begin{proof}
For any $X\in\overline{\cN}$, the decomposition  $X=|X^*|^{1/2}V|X|^{1/2}$ with a unitary
$V$ shows that  
$$
\begin{bmatrix}
|X^*| &X \\
X^*&|X|
\end{bmatrix}
$$
is a positive $\tau$-measurable operator affiliated with $\bM_{2}(\cN)=\cN\otimes\bM_2$. 
Consequently,
$$
\begin{bmatrix}
|X^*| +|Y^*|& X+Y \\
X^* +Y^*&|X|+|Y|
\end{bmatrix}
$$
also belongs to $\overline{\bM_{2}(\cN)}^+$. Let $W$ be the unitary part in $X+Y=W|X+Y|$.
Then
$$
\begin{bmatrix}
-W^*& I
\end{bmatrix}
\begin{bmatrix}
|X^*| +|Y^*|&X+Y \\
X^* +Y^*&|X|+|Y|
\end{bmatrix}
\begin{bmatrix}
-W \\ I
\end{bmatrix}
$$
is in $\overline{\cN}^+$, so that
\begin{equation}\label{triangle-block}
|X+Y| \le \frac{|X|+|Y|+W^*(|X^*|+|Y^*|)W}{2}.
\end{equation}
As $g(t)$ is non-decreasing and convex, in the factor case, \eqref{triangle-block} combined
with \cite[Proposition 4.6\,(ii)]{FK} shows that $g(|X+Y|)$ is spectrally dominated by
$\{g(|X|+|Y|)+W^*g(|X^*|+|Y^*|)W\}/2$. Using Lemma \ref{lemma-sp-dominance} completes the
proof of the proposition when $\cN$ is a factor. The general case follows by using the
measurable selection method as in the previous proofs.
\end{proof}

\section{Extension of symmetric norms and anti-norms}

The aim of this section is to show that a symmetric norm on $\cN$ and a symmetric anti-norm
on $\cN^+$ can naturally be extended, respectively, to $\overline{\cN}$ and to
$\overline{\cN}^+$. First, let $\|\cdot\|$ be a symmetric norm on $\cN$. For each
$X\in\overline{\cN}$ and $s>0$, the function $t\mapsto\beta_s(t):=\min\{s,t\}$ is used to
define $|X|\wedge s:=\beta_s(|X|)$. Since Lemma \ref{monot} implies that
$\|\,|X|\wedge s\|$ is increasing as $s\nearrow\infty$, a natural extension of $\|\cdot\|$
to $\overline{\cN}$ is given as
$$
\|X\|:=\lim_{s\nearrow\infty}\|\,|X|\wedge s\|
=\sup_{s>0}\|\,|X|\wedge s\|\in[0,\infty].
$$

\begin{prop}\label{prop-ext}
The above extension of $\|\cdot\|$ becomes a symmetric norm on $\overline{\cN}$ (permitting
value $\infty$).
\end{prop}

\begin{proof}
It is immediate to see that the extended $\|\cdot\|$ on $\overline{\cN}$ satisfies
$\|\alpha X\|=|\alpha|\,\|X\|$ for all $\alpha\in\bC$ and $\|UXV\|=\|X\|$ for all unitaries
$U,V\in\cN$. For every $A,B\in\overline{\cN}^+$, $s>0$ and $\eps>0$, since $\beta_s$ is
concave on $[0,\infty)$, by Theorem \ref{th-sub} there are unitaries $U,V\in\cN$ such that
$$
(A+B)\wedge s\le U(A\wedge s)U^*+V(B\wedge s)V^*+\eps I.
$$
By Lemma \ref{monot} this implies that
$$
\|(A+B)\wedge s\|\le\|A\wedge s\|+\|B\wedge s\|+\eps\|I\|.
$$
Letting $\eps\searrow0$ and $s\nearrow\infty$ gives $\|A+B\|\le\|A\|+\|B\|$. Next we extend
the monotonicity of $\|\cdot\|$ to $A\le B$ in $\overline{\cN}^+$. For every $s>0$ and
$\eps>0$ there exists a unitary $U\in\cN$ such that $U(A\wedge s)U^*\le B\wedge s+\eps I$.
Since $\lambda_t(A\wedge s)\le\lambda_t(B\wedge s)$ for all $t\in(0,\tau(I))$, this follows
from Lemma \ref{lemma-sp-dominance} when $\cN$ is a factor. For the non-factor case, we can
use the measurable selection method under the central direct decomposition as in the previous
section. Therefore, $\|A\wedge s\|\le\|B\wedge s\|+\eps\|I\|$, which implies that
$\|A\|\le\|B\|$. Now, the subadditivity $\|X+Y\|\le\|X\|+\|Y\|$ in the whole
$\overline{\cN}$ follows by the triangle inequality
$$
|X+Y|\le U|X|U^*+V|Y|V^*
$$
for some unitaries $U,V\in\cN$ (\cite[Lemma 4.3]{FK}, \cite{Ko}), or else by use of
\eqref{triangle-block}.
\end{proof}

Secondly, let $\|\cdot\|_!$ be a symmetric anti-norm on $\cN^+$. For every
$A\in\overline{\cN}^+$ and $s>0$, since $\|A\wedge s\|_!$ increasing as $s\nearrow\infty$,
we can extend $\|\cdot\|_!$ to $\overline{\cN}^+$ as
$$
\|A\|_!:=\lim_{s\nearrow\infty}\|A\wedge s\|_!\in[0,\infty].
$$

\begin{prop}\label{prop-ext-anti}
The above extension of $\|\cdot\|_!$ to $\overline{\cN}^+$ still satisfies the three
conditions $(1)_!$, $(2)_!$ and $(3)_!$.
\end{prop}

\begin{proof}
Since (1)$_!$ and (2)$_!$ are immediate by definition, we may prove (3)$_!$. Let
$A,B\in\overline{\cN}^+$, $s>0$ and $\eps>0$ be arbitrary. We show that there exists a
unitary $U\in\cN$ such that
\begin{equation}\label{cond-2}
(A+B)\wedge2s+\eps I\ge U(A\wedge s+B\wedge s)U^*.
\end{equation}
When $\cN$ is a factor, this follows from Lemma \ref{lemma-sp-dominance} since
we have, for  $t\in(0,\tau(I))$,
$$
\lambda_t((A+B)\wedge2s)=\lambda_t(A+B)\wedge2s
\ge\lambda_t(A\wedge s+B\wedge s)\wedge 2s = \lambda_t(A\wedge s+B\wedge s)
$$
due to $A\wedge s+B\wedge s \le 2sI$.
For the non-factor case, under the decompositions \eqref{decomp-1} and \eqref{decomp-2} we
have $A\wedge s=\int_\Omega^\oplus A_\omega\wedge s\,d\nu(\omega)$ and similarly for
$B\wedge s$ and $(A+B)\wedge2s$.
For each $\omega\in\Omega$, from the above factor case, there is a unitary $V\in\cN_\omega$
such that
$$
(A_\omega+B_\omega)\wedge2s+\eps I_\omega
\ge V(A_\omega\wedge s+B_\omega\wedge s)V^*.
$$
Now, define $F(\omega)$ to be the set of unitaries $V\in\cN_\omega$ satisfying
\eqref{cond-2}, and use the measurable selection method as before to obtain a measurable
field $\omega\mapsto U_\omega$ such that $U_\omega\in F(\omega)$ for all $\omega\in\Omega$. 
Therefore, we have \eqref{cond-2} with the unitary
$U:=\int_\Omega^\oplus U_\omega\,d\nu(\omega)$ in $\cN$, so that
$\|(A+B)\wedge2s+\eps I\|_!\ge\|A\wedge s\|_!+\|B\wedge s\|_!$. Letting $\eps\searrow0$
and $s\nearrow\infty$ gives $\|A+B\|_!\ge\|A\|_!+\|B\|_!$.
\end{proof}

\begin{remark}\label{rem-1}
Let $\cN=\oplus_{n=1}^{\infty} \cN_n$ where $\cN_n=\bM_2$ for all $n$. We equip $\cN$ with
the trace $\tau=\sum_{n\ge 1} 2^{-n}\tau_n$ where $\tau_n$ is the standard trace on $\cN_n$.
The function $t\mapsto \beta_n (t)=\min\{t,n\}$ is concave and there exist
$A_n,B_n\in\bM_2^+$ such that
${\mathrm{Tr\,}}\beta_n(A_n+B_n)<{\mathrm{Tr\,}}\beta_n(A_n)+{\mathrm{Tr\,}}\beta_n(B_n)$.
Hence, there exist $A,B\in\overline{\cN}^+$ such that, for all $n\in\bN$,
$\tau((A+B)\wedge n))<\tau(A\wedge n) +\tau(A\wedge n) $. Such a phenomenon for $\tau$
(regarded as an anti-norm) explains why the proof of the superadditivity part $(3)_!$ in
Proposition \ref{prop-ext-anti} is non-trivial.
\end{remark}

Derived anti-norms extended to $\overline{\cN}^+$ have the following simple properties;
(a) and (b) may be used to check that a symmetric anti-norm is not a derived one.

\begin{prop}\label{prop-deriv}
Let $\|\cdot\|_!$ be a derived anti-norm on $\cN^+$, which is derived from a symmetric norm
$\|\cdot\|$ on $\cN$ and a $p>0$. The following hold for the extensions of $\|\cdot\|_!$
to $\overline{\cN}^+$ and $\|\cdot\|$ to $\overline{\cN}$.
\begin{itemize}
\item[(a)] $\|A\|_!<\infty$ for all $A\in\overline{\cN}^+$.
\item[(b)] If $A\in\overline{\cN}^+$ is singular, i.e., the kernel of $A$ is non-trivial,
then $\|A\|_!=0$.
\item[(c)] If $A\in\overline{\cN}^+$ is nonsingular, then $\|A\|_!=\|A^{-p}\|^{-1/p}$ (with
convention $\infty^{-1/p}=0$).
\end{itemize}
\end{prop}

\begin{proof}
The proof of (a) is easy and left to the reader.

(b)\enspace
Assume that $\ker A\ne\{0\}$, and let $P$ be the projection onto $\ker A$. For each $s>0$
and $\eps>0$ we have $(A\wedge s+\eps I)^{-p}\ge\eps^{-p}P$ and so
$\|(A\wedge s+\eps I)^{-p}\|\ge\eps^{-p}\|P\|$. By definition \eqref{form-deriv},
$$
\|A\wedge s\|_!=\lim_{\eps\searrow0}\|(A\wedge s+\eps I)^{-p}\|^{-1/p}
\le\lim_{\eps\searrow0}\eps\|P\|^{-1/p}=0.
$$
Therefore, $\|A\|_!=0$.

(c)\enspace
First we extend \eqref{form-deriv} to a nonsingular $A\in\overline{\cN}^+$. Let
$A\in\overline{\cN}^+$ with $\ker A=\{0\}$; hence $A^{-p}\in\overline{\cN}^+$. For every
$s>0$ and $\eps>0$, set
$$
\phi_s(\eps):=\sup_{x\ge0}\bigl\{x^{-p}\wedge s-(x+\eps)^{-p}\bigr\}.
$$
Then it is clear that $\phi_s(\eps)>0$ and $\lim_{\eps\searrow0}\phi_s(\eps)=0$ for each
$s>0$. Since $\|(A+\eps I)^{-p}\|$ increases as $\eps\searrow0$ by Lemma \ref{monot} and
$$
A^{-p}\wedge s\le(A+\eps I)^{-p}+\phi_s(\eps)I,
$$
we have $\|A^{-p}\wedge s\|\le\lim_{\eps\searrow0}\|(A+\eps I)^{-p}\|$. Hence
$\|A^{-p}\|\le\lim_{\eps\searrow0}\|(A+\eps I)^{-p}\|$. On the other hand, since
$(A+\eps I)^{-p}\le A^{-p}\wedge\eps^{-p}$, we have
$\|(A+\eps I)^{-p}\|\le\|A^{-p}\|$ for every $\eps>0$. Therefore,
\begin{equation}\label{form-deriv-ext}
\|A^{-p}\|^{-1/p}=\lim_{\eps\searrow0}\|(A+\eps I)^{-p}\|^{-1/p}.
\end{equation}
Now, looking at the function $x\ge0\mapsto(x\wedge s+\eps)^{-p}$, one can easily see that
$$
(A+\eps I)^{-p}\le(A\wedge s+\eps I)^{-p}\le(A+\eps I)^{-p}+s^{-p}I
$$
and hence
$$
\|(A+\eps I)^{-p}\|^{-1/p}\ge\|(A\wedge s+\eps I)^{-p}\|^{-1/p}
\ge(\|(A+\eps I)^{-p}\|+s^{-p}\|I\|)^{-1/p}.
$$
Thanks to \eqref{form-deriv-ext}, letting $\eps\searrow0$ gives
$$
\|A^{-p}\|^{-1/p}\ge\|A\wedge s\|_!\ge(\|A^{-p}\|+s^{-p}\|I\|)^{-1/p}
$$
so that $\|A\|_!=\|A^{-p}\|^{-1/p}$ follows.
\end{proof}

\section{Superadditivity for convex functions}

The aim of this section is to prove the next superadditivity theorem for a symmetric
anti-norm involving a convex function $g(t)$ on $[0,\infty)$. Note that a non-negative
convex function $g(t)$ on $[0,\infty)$ is superadditive if and only if $g(0)=0$. Also,
note that the assumption on $g(t)$ in the theorem is best possible; indeed, the assumption
is necessary even for the classical Rotfel'd trace inequality for matrices.

\vskip 5pt\noindent
\begin{theorem}\label{theorem-1}
Let $A,B\in\overline{\cN}^+$ and let $g(t)$ be a non-negative convex function on $[0,\infty)$
with $g(0)=0$. Then, for all symmetric anti-norms on $\cN^+$,
$$
\|g(A+B)\|_!\ge\|g(A)\|_!+\|g(B)\|_!.
$$
\end{theorem}

\vskip 5pt
Here, and in the whole sequel, we consider that symmetric  anti-norms on $\cN^+$ are
automatically defined on the full cone $\overline\cN^+$ as in Proposition
\ref{prop-ext-anti}, and similarly symmetric norms on $\cN$ are defined on the whole
$\overline{\cN}$ as in Proposition \ref{prop-ext}.

Theorem \ref{theorem-1} claims a numerical inequality; however, its proof relies on an 
operator inequality presented in Section 3. Indeed, when $A,B\in\cN^+$, it is a
straightforward consequence of Theorem \ref{th-super}. In the unbounded case, we cannot
argue by letting $\eps\searrow 0$ in Theorem \ref{th-super} because
$(4)_!$ may not hold in $\overline{\cN}^+$; for instance, when $\|\cdot\|_!$ is the Fuglede-Kadison determinant (see the end of Section 6), there is an $A\in\overline{\cN}^+$ which has the non-trivial kernel (hence $\|A\|_!=0$) but satisfies $\| A+\eps I\|_!=\infty$ for all $\eps>0$.

\begin{proof}
Let $A,B\in\overline{\cN}^+$ and $s>0$ be arbitrary. Since $g(t)$ is continuous and
non-decreasing on $[0,\infty)$, we have thanks to \cite[Lemma 2.5\,(iv)]{FK}
\begin{align*}
\lambda_t(g((A+B)\wedge2s))&=g(\lambda_t(A+B)\wedge2s) \\
&\ge g(\lambda_t(A\wedge s+B\wedge s)) \\
&=\lambda_t(g(A\wedge s+B\wedge s)),
\qquad t\in(0,\tau(I)).
\end{align*}
Given $\eps>0$, when $\cN$ is a finite factor, Lemma \ref{lemma-sp-dominance} then entails
$$
g((A+B)\wedge2s)+\eps I\ge Wg(A\wedge s+B\wedge s)W^*
$$
for some unitary $W\in\cN$. This inequality can be extended to the non-factor case, by
using the measurable selection method under the central direct decomposition as in
Section 3, while full details are left to the reader. Hence, Theorem \ref{th-super} applied
to $A\wedge s$, $B\wedge s$ in place of $A$, $B$ shows that
$$
g((A+B)\wedge2s)+2\eps I\ge Ug(A\wedge s)U^*+Vg(B\wedge s)V^*
$$
for some unitaries $U,V\in\cN$. Therefore,
$$
\|g((A+B)\wedge2s)+2\eps I\|_!\ge\|g(A\wedge s)\|_!+\|g(B\wedge s)\|_!
$$
so that $\|g((A+B)\wedge2s)\|_!\ge\|g(A\wedge s)\|_!+\|g(B\wedge s)\|_!$. Since a simple
estimation gives
$$
\|g(A)\|_!=\lim_{s\nearrow\infty}\|g(A)\wedge s\|_!
=\lim_{s\nearrow\infty}\|g(A\wedge s)\|_!,
$$
the claimed inequality follows.
\end{proof}

In the rest of the section we collect a few special illustrations of Theorem \ref{theorem-1}.

\vskip 5pt\noindent
\begin{cor}\label{new-cor} Let $A,B\in\overline{\cN}^+$ be nonsingular and $p_1,\dots,p_m$ be positive scalars such that $\sum_{i=1}^m p_i\ge 1$. Then, for all symmetric norms on $\cN$,
$$
\prod_{i=1}^m \left\| (A+B)^{-p_i}\right\|^{-1}\ge \prod_{i=1}^m  \left\|  A^{-p_i}\right\|^{-1}+  \prod_{i=1}^m  \left\| B^{-p_i}\right\|^{-1}.
$$
\end{cor}

\vskip 5pt\noindent
\begin{proof} Let $g_i(t)$ be strictly increasing convex functions on $[0,\infty)$ with
$g_i(t)=0$, and let $q_i>0$, $1\le i\le m$. Theorem \ref{theorem-1} applied to the derived
anti-norms $A\mapsto\| A^{-q_i}\|^{-1/q_i}$ yields, in view of Proposition
\ref{prop-deriv}\,(c),

$$
\left\| g_i^{-q_i}(A+B)\right\|^{-1/q_i}
\ge \left\| g_i^{-q_i}(A)\right\|^{-1/q_i} + \left\| g_i^{-q_i}(B)\right\|^{-1/q_i},
\quad 1\le i\le m.
$$
Now, assume that $\sum_{i=1}^m q_i=1$.
By  the elementary inequality from the concavity of the weighted geometric mean,
$$
\prod_{i=1}^m (a_i+b_i)^{q_i}\ge \prod_{i=1}^m  a_i^{q_i}+  \prod_{i=1}^m  b_i^{q_i}
$$
with $a_i= \left\| g_i^{-q_i}(A)\right\|^{-1/q_i}$ and
$b_i=\left\| g_i^{-q_i}(B)\right\|^{-1/q_i}$, we have
\begin{equation}\label{new-eq2}
\prod_{i=1}^m \left\| g_i^{-q_i}(A+B)\right\|^{-1} \ge
\prod_{i=1}^m  \left\| g_i^{-q_i}(A)\right\|^{-1} + \prod_{i=1}^m \left\|  g_i^{-q_i}(B)\right\|^{-1}. 
\end{equation}
Let $p=\sum_{i=1}^m p_i$. Take in \eqref{new-eq2} $g_i(t)=t^{p}$ and $q_i=p_i/p$, $1\le i\le m$, to obtain the required estimate. \end{proof}

\vskip 5pt\noindent
\begin{cor}\label{new-cor2} Let $A,B\in\overline{\cN}^+$ be nonsingular. Then, for all 
symmetric norms on $\cN$, and $m=1,2,\ldots$,
$$
\left\| \sum_{k=1}^m(A+B)^{-k} \right\|^{-1} \ge
\left\| \sum_{k=1}^m A^{-k} \right\|^{-1} +  \left\| \sum_{k=1}^m B^{-k} \right\|^{-1}.
$$
\end{cor}

\vskip 5pt\noindent
\begin{proof} Let $m\ge2$ and let 
$$
g(t):=\frac{t^m}{1+t+\cdots+t^{m-1}}={t^{m+1}-t^m\over t^m-1}.
$$
Then $g(0)=0$ and for $t>0$,
$$
g''(t)={mt^{m-2}\bigl\{(m-1)t^{m+1}-(m+1)t^m+(m+1)t-(m-1)\bigr\}\over(t^m-1)^3}.
$$
For $\phi(t):=(m-1)t^{m+1}-(m+1)t^m+(m+1)t-(m-1)$ compute
\begin{align*}
\phi'(t)&=(m+1)\bigl\{(m-1)t^m-mt^{m-1}+1\bigr\}, \\
\phi''(t)&=(m+1)m(m-1)t^{m-2}(t-1).
\end{align*}
Since $\phi(1)=\phi'(1)=\phi''(1)=0$, we see that $\phi(t)\le0$ for $0\le t\le1$ and
$\phi(t)\ge0$ for $t\ge1$. Hence $g''(t)\ge0$ for all $t>0$, so $g(t)$ is convex on
$(0,\infty)$. 
Note that $g(t)=\bigl(\sum_{k=1}^m t^{-k}\bigr)^{-1}$ for $t>0$.
So, applying Theorem \ref{theorem-1} to this function $g(t)$ and  the derived anti-norm
$A\mapsto \| A^{-1}\|^{-1}$ proves the corollary.
\end{proof}

\vskip 5pt
The next corollary involves an anti-norm specific to the matrix algebra $\bM_n$. Here
$\wedge^m A$ denotes the $m$th antisymmetric tensor power of a matrix $A$ (see
\cite{Bhatia}).

\vskip 5pt\noindent
\begin{cor} Let $A,B\in\bM_n^+$ be nonsingular. Let $g:[0,\infty)\to[0,\infty)$ be a
strictly increasing convex function with $g(0)=0$. Then, for all $0<q\le 1$ and all  
$m=1,2,\ldots$,
$$
\biggl\{\frac{{\mathrm{Tr\,}}\wedge^{m} g^q(A+B)}{{\mathrm{Tr\,}}\wedge^{m-1} g^q(A+B)}
\biggr\}^{1/q} \ge
\biggl\{\frac{{\mathrm{Tr\,}}\wedge^{m} g^q(A)}{{\mathrm{Tr\,}}\wedge^{m-1} g^q(A)}
\biggr\}^{1/q} +
\biggl\{\frac{{\mathrm{Tr\,}}\wedge^{m} g^q(B)}{{\mathrm{Tr\,}}\wedge^{m-1} g^q(B)}
\biggr\}^{1/q}.
$$
\end{cor}

\vskip 5pt
Note that letting $m=q=1$ we recapture the Rotfel'd trace inequality.

\vskip 5pt\noindent
\begin{proof}
By a theorem of Marcus and Lopes \cite{ML} (also \cite[p.\,116]{MOA}), the functional on
positive nonsingular matrices
$$
A\mapsto \frac{ {\mathrm{Tr\,}} \wedge^{m} A}{ {\mathrm{Tr\,}}\wedge^{m-1} A}
$$
is superadditive. This can be extended as an anti-norm on the whole of $\bM_n^+$ by using
condition $(4)_!$. The corollary then follows from Lemma \ref{lem-aux1} combined with
Theorem \ref{theorem-1}.
\end{proof}

\section{Full symmetry  and majorization}

In this section we consider a  stronger symmetry property of norms and anti-norms
in connection with majorization relations. We will focus on the case of diffuse algebras. Indeed, the case of $\bM_n$ is simpler as well as classical. Meanwhile, the setting of a general finite von
Neumann algebra $\cN$ is inappropriate to apply the majorization technique. This issue may be
justified by the fact \cite[Theorem 3.27]{FHNS} that $(\cN,\tau)$ with $\tau(I)=1$ satisfies
the {\it weak Dixmier property} (i.e., $\tau(A)$ is in the $\|\cdot\|_\infty$-closure of
the convex hull of $\{B\in\cN^+:\lambda(B)=\lambda(A)\}$ for every $A\in\cN^+$) if and only
if either $(\cN,\tau)$ is a subalgebra of $(\bM_n,n^{-1}\mathrm{Tr})$ containing all
diagonal matrices, or $\cN$ is diffuse.

Thus, in Sections 6 and 7, we shall always write $\cM$ (differently from $\cN$) to denote a diffuse finite von Neumann algebra with a faithful normal trace $\tau$ such that $\tau(I)=1$.

\begin{definition}\label{def-fully}\rm
A symmetric norm $\|\cdot\|$ on $\cM$ is said to be {\it fully symmetric} (or
{\it rearrangement invariant}\,) if $\lambda(A)=\lambda(B)$ implies $\|A\|=\|B\|$ for
$A,B\in\cM^+$. Also, a symmetric anti-norm $\|\cdot\|_!$ on $\cM^+$ is said to be
{\it fully symmetric} if the same property holds for $\|\cdot\|_!$.
\end{definition}

The next proposition says that the symmetry and the full symmetry properties are equivalent when $\cM$ is a II$_1$ factor (this is
also true and well-known for $\bM_n$).

\begin{prop}\label{prop-factor}
If $\cM$ is a  factor, then any symmetric norm and any symmetric anti-norm are fully
symmetric.
\end{prop}

\begin{proof}
Let $A,B\in\cM^+$ and assume that $\lambda(A)=\lambda(B)$. For a symmetric norm $\|\cdot\|$,
by Lemmas \ref{monot} and \ref{lemma-sp-dominance} we have $\|A\|\le\|B\|+\eps\|I\|$ for
every $\eps>0$, so $\|A\|\le\|B\|$ and the reverse inequality is similar. The proof is
similar for a symmetric anti-norm by using condition (4)$_!$.
\end{proof}

Recall some notions of majorization relevant to our discussions below. For
$A,B\in\overline{\cM}^+$, the {\it submajorization} $A\prec_w B$ is defined as
$\lambda(A)\prec_w\lambda(B)$, i.e.,
$$
\int_0^t\lambda_s(A)\,ds\le\int_0^t\lambda_s(B)\,ds,\qquad t\in(0,1),
$$
and the {\it majorization} $A\prec B$ means that $A\prec_wB$ and $\tau(A)=\tau(B)<\infty$.
The {\it supermajorization} $A\prec^w B$ is defined as
$$
\int_t^1\lambda_s(A)\,ds\ge\int_t^1\lambda_s(B)\,ds,\qquad t\in(0,1).
$$
The {\it log-supermajorization} $A\prec^{w(\log)}B$ is defined as
$$
\int_t^1\log\lambda_s(A)\,ds\ge\int_t^1\log\lambda_s(B)\,ds,\qquad t\in(0,1).
$$
These definitions make sense since the integrals always exist permitting $\pm\infty$.
(The finiteness assumption $\tau(I)<\infty$ is essential to introduce the supermajorization and the log-supermajorization.)

\begin{example}\label{Exam}
(1)\enspace For each $t\in(0,1]$ the functional
$$
\|X\|_{(t)}:=\int_0^t\mu_s(X)\,ds,\qquad X\in\cM
$$
is a fully symmetric norm on $\cM$, that is the continuous version of the Ky Fan $k$-norm
for matrices. The triangle inequality of $\|\cdot\|_{(t)}$ is a consequence of the
submajorization $\mu(X+Y)\prec_w\mu(X)+\mu(Y)$ for $X,Y\in\overline{\cM}$ (see \cite{HN1}).

\medskip
(2)\enspace
For each $t\in[0,1)$ the functional
$$
\|A\|_{\{t\}}:=\int_{1-t}^1\lambda_s(A)\,ds,\qquad A\in\cM^+
$$
is a fully symmetric anti-norm on $\cM^+$. The superadditivity of $\|\cdot\|_{\{t\}}$ is a
consequence of the majorization $\lambda(A+B)\prec\lambda(A)+\lambda(B)$ for $A,B\in\cM^+$
(see \cite{HN1}). This anti-norm is not a derived anti-norm.

\medskip
(3)\enspace
For each $t\in[0,1)$ and $p>0$, the derived anti-norm from the above $\|\cdot\|_{(t)}$ and
$p$ is written as
\begin{equation}\label{deriv-form}
\|A\|_!:=\lim_{\eps\searrow0}\|(A+\eps I)^{-p}\|_{(t)}^{-1/p}
=\biggl(\int_{1-t}^1\lambda_s(A)^{-p}\,ds\biggr)^{-1/p},\qquad A\in\cM^+
\end{equation}
with the usual convention $0^{-p}=\infty$ and $\infty^{-1/p}=0$. Obviously, this derived
anti-norm is fully symmetric. One can easily find a sequence $A_n,A\in\cM^+$ such that
$\|A_n-A\|_\infty\to0$ and $\|A_n\|_!=0$ for all $n$, but $\|A\|_!>0$. Therefore, this
$\|\cdot\|_!$ is not $\|\cdot\|_\infty$-continuous on $\cM^+$. On the other hand, the
derived anti-norm from $\|\cdot\|_\infty$ (and any $p>0$) is $\lambda_1(A)$, which is
$\|\cdot\|_\infty$-continuous on $\cM^+$. Thus, the continuity behavior with respect to
the operator norm in the diffuse case is subtler than the matrix case.
\end{example}

Since $\cM$ is diffuse, we can choose a family $\{F_t\}_{0\le t\le1}$ of projections
as in the proof of Lemma \ref{lemma-sp-dominance}. In our discussions below we will use such
a family $\{F_t\}$ without mentioning explicitly.

\subsection{Fully symmetric norms}

The following properties of fully symmetric norms were more or less discussed in study of non-commutative Banach function spaces (e.g., \cite{DDP1,DDP2,Wa}), usually as working assumptions rather than results. The book \cite{LSZ} contains a nice discussion on this topic.

\begin{prop}\label{full-sym}
Let $\|\cdot\|$ be a fully symmetric norm on $\cM$.
\begin{itemize}
\item[(a)] For every $A,B\in\overline{\cM}^+$, $A\prec_w B$ implies $\|A\|\le\|B\|$.
\item[(b)] If $A,A_n\in\overline{\cM}^+$ and $A_n\nearrow A$ in the $\tau$-measure topology
(or more weakly $\lambda_t(A_n)\nearrow\lambda_t(A)$ for a.e.\ $t\in(0,1)$), then
$\|A_n\|\nearrow\|A\|$.
\item[(c)] $\|X\|_1\|I\|\le\|X\|\le\|X\|_\infty\|I\|$ for all $X\in\overline{\cM}$, where
$\|X\|_1:=\tau(|X|)$ and $\|X\|_\infty:=\lim_{t\searrow0}\lambda_t(|X|)$, the operator norm.
\end{itemize}
\end{prop}

\begin{proof}
(a) (for bounded operators)\enspace
Assume that $A,B\in\cM^+$ and $A\prec_wB$.
Since $f:=\lambda(A)\prec_wg:=\lambda(B)$ as functions in $L^\infty(0,1)^+$, by
\cite[Theorem 1.1]{Ch} there is an $h\in L^\infty(0,1)^+$ such that $f\le h\prec g$. Let
$\Gamma$ denote the set of bijective Borel transformations on $(0,1)$ preserving the
Lebesgue measure. In the same way as in the proof of \cite[Theorem 2.1, Lemma 2.2]{HN3} we
can see that the decreasing rearrangement $h^*$ of $h$ is in the $\|\cdot\|_\infty$-closure
of the convex hull of $\{\gamma f:\gamma\in\Gamma\}$, where $(\gamma f)(t):=f(\gamma^{-1}t)$.
Thus, for every $\eps>0$ there are $\gamma_i\in\Gamma$ and $\lambda_i>0$, $1\le i\le k$,
such that $\sum_{i=1}^k\lambda_i=1$ and $h^*\le\sum_{i=1}^k\lambda_i\gamma_ig+\eps1$. Let
$C:=\int_0^1h^*(t)\,dF_t$ and $B_i:=\int_0^1(\gamma_ig)(t)\,dF_t$. Since
$\lambda(B_i)=\lambda(B)$, the monotonicity and the full symmetry of $\|\cdot\|$ yield
$$
\|A\|\le\|C\|\le\sum_{i=1}^k\lambda_i\|B_i\|+\eps\|I\|=\|B\|+\eps\|I\|.
$$
Hence $\|A\|\le\|B\|$.

(b)\enspace
First we prove that if $A,A_n\in\cM^+$ and $\lambda_t(A_n)\nearrow\lambda_t(A)$ a.e., then
$\|A_n\|\nearrow\|A\|$. By (a), $\|A_n\|$ is increasing in $n$. For every $\eps>0$, since
$\lambda_0(A_n)=\|A_n\|_\infty\nearrow\lambda_0(A)=\|A\|_\infty$ and
$$
\lim_{t\searrow0}{1\over t}\int_0^t(\lambda_s(A)-\lambda_s(A_n))\,ds
=\lambda_0(A)-\lambda_0(A_n),
$$
one can choose $n_0\in\bN$ and $\delta>0$ so that
$t^{-1}\int_0^t(\lambda_s(A)-\lambda_s(A_{n_0}))\,ds<\eps$ for all $t\in(0,\delta)$,
and hence
$$
{1\over t}\int_0^t(\lambda_s(A)-\lambda_s(A_n))\,ds<\eps,
\qquad t\in(0,\delta),\ n\ge n_0.
$$
Furthermore, for any $t\in[\delta,1)$ one has
$$
{1\over t}\int_0^t(\lambda_s(A)-\lambda_s(A_n))\,ds
\le{1\over\delta}\int_0^1(\lambda_s(A)-\lambda_s(A_n))\,ds\searrow0
\quad\mbox{as $n\to\infty$}
$$
by the dominated convergence theorem. Hence there exists an $n_1\in\bN$ such
that
$$
{1\over t}\int_0^t(\lambda_s(A)-\lambda_s(A_n))\,ds<\eps,
\qquad t\in[\delta,1),\ n\ge n_1.
$$
If $n\ge\max\{n_0,n_1\}$, then the above estimates imply that $A\prec_w A_n+\eps I$
so that $\|A\|\le\|A_n\|+\eps\|I\|$ by (a). Thus $\|A_n\|\nearrow\|A\|$.

Next assume that $A,A_n\in\overline{\cM}^+$ and $\lambda_t(A_n)\nearrow\lambda_t(A)$ a.e.
For every $s>0$, since $\lambda_t(A_n\wedge s)\nearrow\lambda_t(A\wedge s)$ a.e., we have
$\|A_n\wedge s\|\nearrow\|A\wedge s\|$ from the first step. Hence
$\|A_n\|=\sup_{s>0}\|A_n\wedge s\|$ is increasing in $n$ and
$$
\|A\|=\sup_{s>0}\|A\wedge s\|=\sup_{s>0,\,n\in\bN}\|A_n\wedge s\|=\sup_{n\in\bN}\|A_n\|.
$$
Therefore, $\|A_n\|\nearrow\|A\|$.

(c)\enspace
For $X\in\cM$, $\|X\|\le\|X\|_\infty\|I\|$ was given in \eqref{norm-ineq}. Since
$\tau(|X|)I\prec|X|$, $\|X\|_1\|I\|\le\|X\|$ by (a). These can easily extend to all
$X\in\overline{\cM}$ by (b).

(a) (for unbounded operators)\enspace
Let $A,B\in\overline{\cM}^+$ and assume $A\prec_wB$. We may assume that $\|B\|<\infty$ and
so $\|A\|_1\le\|B\|_1<\infty$ by (c). Fix $0<\rho<1$. Then, for each $n\in\bN$, since $\int_0^t\lambda_s(B)\wedge m\,ds\nearrow\int_0^t\lambda_s(B)\,ds$ as $m\to\infty$ uniformly in $t\in(0,1)$, one can choose an $m\in\bN$
such that $\rho A\wedge n\prec_wB\wedge m$ and hence $\|\rho A\wedge n\|\le\|B\|$ by (a). Hence $\|A\|\le\|B\|$ follows by letting $n\to\infty$ and then $\rho\nearrow1$.
\end{proof}

In view of Example \ref{Exam}\,(1) we have

\begin{cor}
Let $X,Y\in\overline{\cM}$. Then $|X|\prec_w|Y|$ (i.e., $\mu(X)\prec_w\mu(Y)$) if and only
if $\|X\|\le\|Y\|$ for all fully symmetric norms on $\cM$ (extended to $\overline{\cM}$).
\end{cor}

\subsection{Fully symmetric anti-norms}

Fully symmetric anti-norms have the following properties. It would be worthwhile to consider
these properties in parallel to those in Proposition \ref{full-sym}.

\begin{prop}\label{full-sym-anti}
Let $\|\cdot\|_!$ be a fully symmetric anti-norm on $\cM^+$.
\begin{itemize}
\item[(a)] For every $A,B\in\overline{\cM}^+$, $A\prec^wB$ implies $\|A\|_!\ge\|B\|_!$.
\item[(b)] If $A,A_n\in\cM^+$ and $A_n\searrow A$ (or more weakly
$\lambda_t(A_n)\searrow\lambda_t(A)$ for a.e.\ $t\in(0,1)$), then $\|A_n\|_!\searrow\|A\|_!$.
\item[(c)] $\lambda_1(A)\|I\|_!\le\|A\|_!\le\tau(A)\|I\|_!$ for all $A\in\overline{\cM}$,
where $\lambda_1(A):=\lim_{t\nearrow1}\lambda_t(A)$.
\end{itemize}
\end{prop}

\begin{proof}
(a)\enspace
Since $A\prec^wB$ implies $A\wedge s\prec^wB\wedge s$ for any $s>0$ (seen as in the discrete case \cite[p.\,167]{MOA}), we may consider the
case $A,B\in\cM^+$. Since $f:=\lambda(A)\prec^wg:=\lambda(B)$ as functions in
$L^\infty(0,1)^+$, there is an $h\in L^\infty(0,1)^+$ such that $f\ge h\prec g$. The
remaining proof being similar to that of Proposition \ref{full-sym}\,(a), we omit the
details.

(b)\enspace
Assume that $A,A_n\in\cM$ and $\lambda_t(A_n)\searrow\lambda_t(A)$ for a.e.\ $t\in(0,1)$.
By (a), $\|A_n\|_!$ is decreasing in $n$. For every $\eps>0$, since
$\lambda_1(A_n)\searrow\lambda_1(A)$ and
$$
\lim_{t\searrow0}{1\over t}\int_{1-t}^1(\lambda_s(A_n)-\lambda_s(A))\,ds
=\lambda_1(A_n)-\lambda_1(A),
$$
one can show that $A+\eps I\prec^wA_n$ for all sufficiently large $n$, as in the proof of
Proposition \ref{full-sym}\,(b) by replacing $t^{-1}\int_0^t$ with $t^{-1}\int_{1-t}^1$.
By (a) this implies that $\|A+\eps I\|_!\ge\lim_n\|A_n\|_!$. Letting
$\eps\searrow0$ gives $\|A\|_!\ge\lim_n\|A_n\|_!$ and so $\|A_n\|_!\searrow\|A\|_!$.

(c)\enspace
For $A\in\overline{\cM}^+$, $\lambda_1(A)I\le A$ implies $\lambda_1(A)\|I\|_!\le\|A\|_!$.
Assuming $\tau(A)<\infty$, we have $\tau(A)\|I\|_!\ge\|A\|_!$ from $\tau(A)I\prec A$.
\end{proof}

In view of Example \ref{Exam}\,(2) we have

\begin{cor}
Let $A,B\in\overline{\cM}^+$. Then $A\prec^wB$ if and only if $\|A\|_!\le\|B\|_!$ for all
fully symmetric anti-norms on $\cM^+$ (extended to $\overline{\cM}^+$).
\end{cor}

\begin{remark}
Proposition \ref{full-sym}\,(b) means that a fully symmetric norm extended to
$\overline{\cM}$ satisfies the {\it Fatou property} (see \cite{DDP2}). Proposition
\ref{full-sym-anti}\,(b) is considered as the ``anti-Fatou property".
Even though $(4)_!$ may not hold in $\overline{\cM}^+$ as noted in Section 5,
it is not known whether the anti-Fatou property holds for $A,A_n\in\overline{\cM}^+$ when
$\|A_n\|_!<\infty$ and $A_n\searrow A$ in the $\tau$-measure topology. For fully symmetric
derived anti-norms, this property will be shown in the next subsection.
\end{remark}

\subsection{Fully symmetric derived anti-norms}

In the rest of the section we will consider fully symmetric derived anti-norms.

\begin{lemma}\label{lemma-fully-deriv}
Let $\|\cdot\|_!$ be a derived anti-norm on $\cM^+$, which is derived from a symmetric norm
$\|\cdot\|$ on $\cM$ and a $p>0$. Then $\|\cdot\|_!$ is fully symmetric if and only if so is
$\|\cdot\|$.
\end{lemma}

\begin{proof}
From \eqref{norm-ineq}, a symmetric norm on $\cM$ is fully symmetric if the condition in
Definition \ref{def-fully} holds for invertible $A,B\in\cM^+$. From condition (4)$_!$, the
same is true for a symmetric anti-norm. Hence the assertion follows from the relations
$\|A\|_!=\|A^{-p}\|^{-1/p}$ and $\|A\|=\|A^{-1/p}\|_!^{-p}$ for invertible $A\in\cM^+$.
\end{proof}

Properties (a) and (b) of Proposition \ref{full-sym-anti} are strengthened for fully
symmetric derived anti-norms as follows.
It is worth noting that $A\prec^{w(\log)}B$ is weaker than $A\prec^wB$ for
$A,B\in\overline{\cM}^+$.

\begin{prop}\label{full-sym-deriv}
Let $\|\cdot\|_!$ be a fully symmetric derived anti-norm on $\cM^+$.
\begin{itemize}
\item[(a)] For every $A,B\in\overline{\cM}^+$, $A\prec^{w(\log)}B$ implies
$\|A\|_!\ge\|B\|_!$.
\item[(b)] If $A,A_n\in\overline{\cM}^+$ and $A_n\searrow A$ in the $\tau$-measure topology
(or more weakly $\lambda_t(A_n)\searrow\lambda_t(A)$ for a.e.\ $t\in(0,1)$), then
$\|A_n\|_!\searrow\|A\|_!$. In particular, $(4)_!$ holds in $\overline{\cM}^+$.
\end{itemize}
\end{prop}

\begin{proof}
From Lemma \ref{lemma-fully-deriv}, let $\|\cdot\|_!$ be derived from a fully symmetric
norm $\|\cdot\|$ and a $p>0$. 

(a)\enspace
Assume that $A\prec^{w(\log)}B$. Since this implies that $A\wedge s\prec^{w(\log)}B\wedge s$
for all $s>0$ (similarly to the assertion for $\prec^w$ in the proof of Proposition \ref{full-sym-anti}, by considering the function $\log(e^x\wedge s)$), it is enough to assume that $A,B\in\cM^+$. Furthermore, by replacing $A$ with
$A+\eps I$ for any $\eps>0$, $A$ may be assumed invertible. First, assume that
$\int_0^1\log\lambda_s(B)\,ds=-\infty$, and we prove that $\|B\|_!=0$ for every derived
anti-norm $\|\cdot\|_!$. By Proposition \ref{full-sym}\,(c) it suffices to prove this for
the derived anti-norm from $\|\cdot\|_1$ and any $p>0$. So we may show that
$\lim_{\eps\searrow0}\|(B+\eps I)^{-p}\|_1=\infty$.
Since $\int_0^1\log\lambda_s(B)\,ds=-\infty$ implies that $\lambda_1(B)=0$, we have
$-\log\lambda_s(B)\le\lambda_s(B)^{-p}$ for all $s$ sufficiently near $1$. Hence by
\eqref{deriv-form} for $t=1$,
$$
\lim_{\eps\searrow0}\|(B+\eps I)^{-p}\|_1=\int_0^1\lambda_s(B)^{-p}\,ds=\infty.
$$

Next assume that $B$ as well as $A$ is invertible. Then $A\prec^{w(\log)}B$ means
$\log\lambda(A)\prec^w\log\lambda(B)$. For every $p>0$, since
$\log\lambda(A^{-p})\prec_w\log\lambda(B^{-p})$, we have $A^{-p}\prec_wB^{-p}$ by
\cite[Proposition 1.2]{HN1}. Hence by Proposition \ref{full-sym}\,(a),
$\|A^{-p}\|\le\|B^{-p}\|$
for any fully symmetric norm $\|\cdot\|$, implying the assertion.

Finally, assume that $B$ is not invertible but $\int_0^1\log\lambda_s(B)\,ds>-\infty$, so
$\lambda_s(B)>0$ for all $s\in(0,1)$ and $\log\lambda(B)$ is integrable on $(0,1)$.
One can fix an $r_0\in(0,1)$ such that $\lambda_{r_0}(B)\le\lambda_1(A)$ (since $A$ is
assumed invertible while $B$ is not). For every $\delta>0$ there exists an $r\in(r_0,1)$
such that
$$
\int_r^1(\log\lambda_r(B)-\log\lambda_s(B))\,ds\le\delta(1-r_0),
$$
so we define
$$
\hat B:=\int_0^r\lambda_s(B)\,dF_s+\int_r^1\lambda_r(B)\,dF_s,
$$
which is invertible. If $t\in(r_0,1)$, then
$$
\int_t^1(\log\lambda_s(A)+\delta)\,ds\ge\int_t^1\log\lambda_s(A)\,ds
\ge\int_t^1\log\lambda_{r_0}(B)\,ds\ge\int_t^1\log\lambda_s(\hat B)\,ds.
$$
If $t\in(0,r_0]$, then
\begin{align*}
\int_t^1(\log\lambda_s(A)+\delta)\,ds&\ge\int_t^1\log\lambda_s(B)\,ds+\delta(1-r_0) \\
&=\int_t^1\log\lambda_s(\hat B)\,ds-\int_r^1(\log\lambda_r(B)-\log\lambda_s(B))\,ds
+\delta(1-r_0) \\
&\ge\int_t^1\log\lambda_s(\hat B)\,ds.
\end{align*}
The above estimates imply that $e^\delta A\prec^{w(\log)}\hat B$. Since $\hat B$ is
invertible,
$\|(e^\delta A)^{-p}\|\le\|\hat B^{-p}\|$ as in the previous case. Therefore,
$$
e^\delta\|A^{-p}\|^{-1/p}\ge\|\hat B^{-p}\|^{-1/p}=\|\hat B\|_!\ge\|B\|_!
$$
by Proposition \ref{full-sym-anti}\,(a) since $\lambda(\hat B)\ge\lambda(B)$. Letting
$\delta\searrow0$ gives $\|A\|_!\ge\|B\|_!$.

(b)\enspace
If $\ker A\ne\{0\}$, then there is a $\delta\in(0,1)$ such that $\lambda_{1-\delta}(A)=0$.
Since $\lambda_{1-\delta/2}(A_n)\searrow0$, for every $\eps>0$ there exists an $n_0\in\bN$
such that $\lambda_{1-\delta/2}(A_n)<\eps$ for all $n\ge n_0$. For each $n\ge n_0$, letting
$P_n:=\mathbf{1}_{[0,\eps]}(A_n)$ we have $\tau(P_n)\ge\delta/2$ and
$(A_n+\eps I)^{-p}\ge(2\eps)^{-p}P_n$. Therefore,
$$
\|(A_n+\eps I)^{-p}\|\ge(2\eps)^{-p}\|P_n\|\ge(2\eps)^{-p}\|I\|\delta/2
$$
so that $\|A_n\|_!\le2\eps(\|I\|\delta/2)^{-1/p}$ for all $n\ge n_0$. Hence
$\|A_n\|_!\searrow0$.

Next assume that $\ker A=\{0\}$, so $A^{-p}\in\overline{\cM}^+$. Since
$$
\lambda_t(A_n^{-p})=\lambda_{1-t}(A_n)^{-p}
\nearrow\lambda_{1-t}(A)^{-p}=\lambda_t(A^{-p})
$$
for a.e.\ $t\in(0,1)$, we have $\|A_n^{-p}\|\nearrow\|A^{-p}\|$ by Proposition
\ref{full-sym}\,(b). Therefore, $\|A_n\|_!\searrow\|A\|_!$ by Proposition
\ref{prop-deriv}\,(c).
\end{proof}

The next theorem is concerned with the converse implication of Proposition
\ref{full-sym-deriv}\,(a).

\begin{theorem}\label{theorem-full-deriv}
Let $A,B\in\overline{\cM}^+$ and assume that $\int_0^1\lambda_s(B)^{-p}\,ds<\infty$ for
some $p>0$ (in particular, this is the case if $B\ge\delta I$ for some $\delta>0$). Then
the following two conditions are equivalent:
\begin{itemize}
\item[(i)] $A\prec^{w(\log)}B$;
\item[(ii)] $\|A\|_!\ge\|B\|_!$ for every fully symmetric derived anti-norm $\|\cdot\|_!$ on $\cM^+$.
\end{itemize}
\end{theorem}

To prove the theorem, we first give a lemma. When $A$ is invertible, the lemma is
\cite[Lemma 4.3.6]{Ar} with a simpler proof.

\begin{lemma}\label{lemma-limit}
Let $A\in\cM^+$ and assume that $\int_0^1\lambda_s(A)^{-p}\,ds<\infty$ for some $p>0$.
Then, for every $t\in(0,1]$,
$$
\exp\biggl({1\over t}\int_{1-t}^1\log\lambda_s(A)\,ds\biggr)
=\lim_{p\searrow0}\biggl({1\over t}\int_{1-t}^1\lambda_s(A)^{-p}\,ds\biggr)^{-1/p}
=\lim_{p\searrow0}t^{1/p}\|A^{-p}\|_{(t)}^{-1/p},
$$
where $\|\cdot\|_{(t)}$ is in Example \ref{Exam}\,(1).
\end{lemma}

\begin{proof}
Replacing $A$ with $\alpha A$ for some $\alpha>0$, we may suppose that
$A\le I$. Assume that $\int_0^1\lambda_s(A)^{-p_0}\,ds<\infty$ for some $p_0>0$; then
$0<\lambda_s(A)\le1$ for all $s\in(0,1)$ and $-\lambda_s(A)^{-p}\log\lambda_s(A)$ is
integrable on $(0,1)$ for every
$p\in(0,p_0)$. Write $\ffi(s,p):=\lambda_s(A)^{-p}$ for $s\in(0,1)$ and $p\in(0,p_0)$.
Since
$$
{\ffi(s,p)-\ffi(s,0)\over p}=\partial_p\ffi(s,\theta p)
=-\lambda_s(A)^{-\theta p}\log\lambda_s(A)\le-\lambda_s(A)^{-p}\log\lambda_s(A),
$$
where $\theta\in(0,1)$ (depending on $s,p$) and
$$
\lim_{p\searrow0}{\ffi(s,p)-\ffi(s,0)\over p}=\partial_p\ffi(s,0)=-\log\lambda_s(A),
$$
it follows from the Lebesgue convergence theorem that, for every $t\in(0,1]$,
$$
{d\over dp}\int_{1-t}^1\ffi(s,p)\,ds\bigg|_{p=+0}
=\lim_{p\searrow0}\int_{1-t}^1{\ffi(s,p)-\ffi(s,0)\over p}\,ds
=-\int_{1-t}^1\log\lambda_s(A)\,ds,
$$
where ${d\over dp}(\cdot)\big|_{p=+0}$ means the right derivative at $p=0$. Therefore,
\begin{align*}
\lim_{p\searrow0}\biggl[-{1\over p}\log
\biggl({1\over t}\int_{1-t}^1\lambda_s(A)^{-p}\,ds\biggr)\biggr]
&=-{d\over dp}\log\biggl({1\over t}\int_{1-t}^1\lambda_s(A)^{-p}\,ds\biggr)
\bigg|_{p=+0} \\
&={1\over t}\int_{1-t}^1\log\lambda_s(A)\,ds,
\end{align*}
which is equivalent to the desired limit formula.
\end{proof}

We now prove the theorem.

\begin{proof} (Theorem \ref{theorem-full-deriv})\quad
By Proposition \ref{full-sym-deriv}\,(a) we may prove that (ii) $\Rightarrow$ (i), so assume
that $\int_0^1\lambda_s(B)^{-p}\,ds<\infty$ for some $p>0$ and $\|A\|_!\ge\|B\|_!$ for all
fully symmetric derived anti-norms. It suffices to show that, for each $t\in(0,1)$ fixed,
if $\|A^{-p}\|_{(t)}^{-1/p}\ge\|B^{-p}\|_{(t)}^{-1/p}$ for all $p>0$, then
$\int_{1-t}^1\log\lambda_s(A)\,ds\ge\int_{1-t}^1\log\lambda_s(B)\,ds$ holds. Since all the
relevant quantities depend on $\lambda(A),\lambda(B)$ restricted on $(1-t,1)$, we may assume
that $A,B\in\cM^+$, by replacing $A$, $B$ with $A\wedge\alpha$, $B\wedge\beta$ where
$\alpha:=\lambda_{1-t}(A)$, $\beta:=\lambda_{1-t}(B)$, respectively. Then for every
$\delta>0$ we have
$\|(A+\delta I)^{-p}\|_{(t)}^{-1/p}\ge\|A^{-p}\|_{(t)}^{-1/p}\ge\|B^{-p}\|_{(t)}^{-1/p}$.
Applying Lemma \ref{lemma-limit} to $A+\delta I$ and $B$ yields
$$
\int_{1-t}^1\log(\lambda_s(A)+\delta)\,ds\ge\int_{1-t}^1\log\lambda_s(B)\,ds.
$$
Letting $\delta\searrow0$ gives the desired inequality.
\end{proof}

\begin{remark}\label{R-3.14}
Note that there is a $B\in\cM^+$ such that $\int_0^1\log\lambda_s(B)\,ds>-\infty$ but
$\int_0^1\lambda_s(B)^{-p}\,ds=\infty$ for all $p>0$. For instance, this is the case when
$\lambda_s(B)=\exp\bigl(-1/\sqrt{1-s}\bigr)$. For such a $B\in\cM^+$ and every $p>0$, if
$\|\cdot\|$ is a fully symmetric norm on $\cM^+$, then by Proposition \ref{full-sym}\,(c)
we have
$$
\lim_{\eps\searrow0}\|(B+\eps I)^{-p}\|
\ge\lim_{\eps\searrow0}\|(B+\eps I)^{-p}\|_1\|I\|
=\int_0^1\lambda_s(B)^{-p}\,ds\,\|I\|=\infty.
$$
This means that $\|B\|_!=0$ for every fully symmetric derived anti-norm, so (ii) of Theorem
\ref{theorem-full-deriv} is satisfied for any $A\in\cM^+$. Therefore, (ii) $\Rightarrow$ (i)
does not hold for general $A,B\in\cM^+$. Such a subtle difference between the two conditions
(i) and (ii) never occurs in the matrix case: In the matrix algebra $\bM_n$, the conditions (i) and (ii) in Theorem \ref{theorem-full-deriv} are equivalent; (i) $\Rightarrow$ (ii) is shown in \cite[Lemma 4.10]{BH2}, and (ii) $\Rightarrow$ (i) is implicit in \cite[Example 4.5]{BH2}, the discrete version of the anti-norms in \eqref{F-3.7}.
\end{remark}

\vskip 5pt
The following is a consequence of Lemma \ref{lemma-limit}.

\begin{cor}\label{Delta}
For every $t\in(0,1]$ the functional
\begin{equation}\label{F-3.7}
\Delta_t(A):=\exp\biggl({1\over t}\int_{1-t}^1\log\lambda_s(A)\,ds\biggr),
\qquad A\in\cM^+
\end{equation}
is a symmetric anti-norm on $\cM^+$.
\end{cor}

\begin{proof}
The properties (1)$_!$ and (2)$_!$ of Definition \ref{Def-1} for $\Delta_t$ are clear and
(4)$_!$ is immediate from the monotone convergence theorem. To show (3)$_!$, we may assume
in view of (4)$_!$ that $A,B\in\cM^+$
are invertible. Then Lemma \ref{lemma-limit} yields
\begin{equation}\label{limit-form}
\Delta_t(A)=\lim_{p\searrow0}t^{-1/p}\|A^{-p}\|_{(t)}^{-1/p}
\end{equation}
and the same expressions for $\Delta_t(B)$ and $\Delta_t(A+B)$. Hence (3)$_!$ for $\Delta_t$
follows from that of the derived anti-norms $\|A^{-p}\|_{(t)}^{-1/p}$.
\end{proof}

The symmetric anti-norms $\Delta_t$ are not derived ones, but \eqref{limit-form} says that
they are in the boundary of the derived anti-norms. In particular, when $t=1$,
$$
\Delta(X):=\exp\biggl(\int_0^1\log\lambda_s(|X|)\,ds\biggr)=\Delta_1(|X|),
\qquad X\in\cM
$$
is the {\it Fuglede-Kadison determinant} \cite{FuKa}. This is extended to $\overline{\cM}$
by Proposition \ref{prop-ext-anti} and the above expression holds whenever
$\int_0^1\log\lambda_s(|X|)\,ds$ makes sense permitting $\pm\infty$.
The determinant $\Delta$ has been useful in the non-commutative $H^\infty$ theory
(e.g., \cite{Ar,BleLab}).

\section{Superadditivity with more functions}

For a fully symmetric derived anti-norm, Theorem \ref{theorem-1} for convex functions can
be extended to a considerably larger class of superadditive functions given as follows:

\vskip 5pt\noindent
{\it Let $\cS$ be the class of functions $\psi:[0,\infty)\to[0,\infty)$ such that
$
\psi(t)=f \circ g (t)
$
for some superadditive log-concave function $f:[0,\infty)\to[0,\infty)$ and some
superadditive convex function $g:[0,\infty)\to[0,\infty)$.
}

\vskip 5pt
Recall that $f:[0,\infty)\to [0,\infty)$ is log-concave if 
$
f\left((a+b)/2\right) \ge \sqrt{f(a)f(b)}
$
for all $a,b\ge 0$, i.e., if $\log f:[0,\infty)\to[-\infty,\infty)$ is concave. A convex
function $g:[0,\infty)\to [0,\infty)$ is superadditive if and only if $g(0)=0$. Note that
any function $\psi(t)$ in $\cS$ is superadditive and non-decreasing on $[0,\infty)$ with
$\psi(0)=0$.

Any superadditive log-concave function $\psi:[0,\infty)\to[0,\infty)$ is in $\mathcal{S}$. 
Any convex function  $\psi:[0,\infty)\to[0,\infty)$ with $\psi(0)=0$ is in $\mathcal{S}$.
The next  examples  point out some functions in ${\mathcal{S}}$  which are in the intersection of these two subclasses, or only in one subclass, or in none of them. We implicitly assume that superadditive functions are defined on $[0,\infty)$.

\vskip 5pt\noindent
\begin{example} \begin{itemize} \item The power functions $t\mapsto t^p$, $p\ge 1$, and
the angle function at any $\alpha >0$, $t\mapsto(t-\alpha)_+:=\max\{t-\alpha,0\}$,
are superadditive, convex and log-concave. The function $t\mapsto t\arctan t$
is also superadditive, convex and log-concave.
\item
For any $\gamma >1$, the functions $t\mapsto \sinh t^{\gamma}$ and
$t\mapsto t\exp t^{\gamma}$ are superadditive and convex, but not log-concave.
\item
When $1\le \alpha < \beta$, the function $t\mapsto \min\{ t^{\alpha}, t^{\beta}\}$ is
superadditive and log-concave, but not convex. The function
$t\mapsto t^{\alpha} \exp(-1/t^\beta)$ is the same whenever $\alpha\ge1$ and
$\beta>2\alpha-1+2\sqrt{\alpha(\alpha-1)}$.
When $0<a<b$, the function $t\mapsto(t-a) {\mathbf{1}}_{[b,\infty)}(t)$ is also the same
though not continuous.
\item
For $f(t)=\min\{ t^{\alpha}, t^{\beta}\}$ with $1\le\alpha<\beta$ and $g(t)=\sinh t^\gamma$
or $t\exp t^\gamma$ with $\gamma>1$, $f\circ g(t)$ is a function in $\cS$, but neither log-concave nor convex.
\end{itemize}
\end{example}

\vskip 5pt
Recall that $\cM$ stands for a (finite) diffuse algebra. The superadditivity results in this
section also hold with $\bM_n$ in place of $\cM$ with similar though simpler proofs. The next
theorem is the main result of this section.

\begin{theorem}\label{theorem-2}
Let $A,B\in\overline{\cM}^+$ and let $\psi(t)$ be a function in ${\mathcal{S}}$. Then, for
any fully symmetric derived anti-norm on $\overline{\cM}^+$,
$$
\| \psi(A+B) \|_! \ge \| \psi(A) \|_! + \| \psi(B) \|_!.
$$
\end{theorem}

\vskip 5pt
The proof is based on Theorem \ref{th-super} and the next lemma.

\begin{lemma}\label{lemma-log}
Let $\|\cdot\|_!$ be a fully symmetric derived anti-norm on $\cM^+$. Let
$A,B\in\overline{\cM}^+$. If $f:[0,\infty)\to[0,\infty)$ is superadditive and log-concave,
then
$$
\|f(A+B)\|_!\ge\|f(A)\|_!+\|f(B)\|_!.
$$
\end{lemma}

\begin{proof}
For each $s>0$, since $\lambda(f(A+B))\ge\lambda(f(A\wedge s+B\wedge s))$, by Proposition
\ref{full-sym-deriv}\,(a) we have $\|f(A+B)\|_!\ge\|f(A\wedge s+B\wedge s)\|_!$. Since
$\|f(A)\|_!=\lim_{s\nearrow\infty}\|f(A\wedge s)\|_!$ and similarly for $\|f(B)\|_!$ as in
the proof of Theorem \ref{theorem-1}, we may assume that $A,B\in\cM^+$. We prove that
\begin{equation}\label{log-super}
f(A+B)\prec^{w(\log)}\int_0^1f(\lambda_t(A)+\lambda_t(B))\,dF_t.
\end{equation}
As $f(t)$ is non-decreasing, we have $\lambda(f(A+B))\ge\lambda(f(A))$ so that
$\|f(A+B)\|_!\ge\|f(A)\|_!$, and similarly $\|f(A+B)\|_!\ge\|f(B)\|_!$. Hence, the claimed
inequality is obvious if $\|f(A)\|_!=0$ or $\|f(B)\|_!=0$. So assume that $\|f(A)\|_!>0$ and
$\|f(B)\|_!>0$. This implies by Proposition \ref{prop-deriv}\,(b) that $f(A)$ and $f(B)$ are
nonsingular, so $f(\lambda_t(A))>0$ and $f(\lambda_t(B))>0$ for all $t\in(0,1)$. Hence
$f(t)>0$ for all $t>t_0:=\min\{\lambda_1(A),\lambda_1(B)\}$. Furthermore, thanks to
Proposition \ref{full-sym-deriv}\,(b), it suffices to prove the claimed inequality for
$A+\eps I$ and $B+\eps I$ for any $\eps>0$. Thus we can assume that $f(t)>0$ for all
$t\ge t_0$.

Then, from the majorization $\lambda(A+B)\prec\lambda(A)+\lambda(B)$ and the concavity of
$\log f$, we have
$$
\log\lambda(f(A+B))=\log f(\lambda(A+B))\prec^w\log f(\lambda(A)+\lambda(B)),
$$
which means that the log-supermajorization \eqref{log-super} holds. As $f(t)$ is further
superadditive, \eqref{log-super} entails
$$
f(A+B)\prec^{w(\log)}\int_0^1(f(\lambda_t(A))+f(\lambda_t(B)))\,dF_t.
$$
Therefore, Proposition \ref{full-sym-deriv}\,(a) implies that
\begin{align*}
\|f(A+B)\|_!&\ge\bigg\|\int_0^1f(\lambda_t(A))\,dF_t\bigg\|_!
+\bigg\|\int_0^1f(\lambda_t(B))\,dF_t\bigg\|_! \\
&=\|f(A)\|_!+\|f(B)\|_!.
\end{align*}
\end{proof}

\begin{proof} (Theorem \ref{theorem-2})\quad
As in the proof of the previous lemma, we may and do assume that $A,B\in\cM^+$. Let $\psi(t)$
in $\cS$ be written as $\psi(t)=f\circ g(t)$, with $f(t)$ superadditive log-concave and
$g(t)$ superadditive convex. If $\|\psi(A)\|_!=0$ or $\|\psi(B)\|_!=0$, then the claimed
inequality follows as in the proof of the previous lemma. So we may assume that
$\|\psi(A)\|_!>0$ and $\|\psi(B)\|_!>0$. This implies that $f(\lambda_t(g(A)))>0$ and
$f(\lambda_t(g(B)))>0$ for all $t\in(0,1)$.  Hence, $f(t)>0$ for all
$t>t_0:=\min\{\lambda_1(g(A)),\lambda_1(g(B))\}$, so that $f(t)$ is continuous on
$(t_0,\infty)$. For every $\eps>0$ let $U,V\in\cM$ be unitaries as given in Theorem
\ref{th-super}. Then we have
$$
\lambda_t(f(g(A+B)+\eps I))\ge\lambda_t(f(Ug(A)U^*+Vg(B)V^*))
$$
for all $t\in(0,1)$. By Proposition \ref{full-sym-deriv}\,(a) and Lemma \ref{lemma-log}, this implies that
\begin{align*}
\|f(g(A+B)+\eps I)\|_!&\ge\|f(Ug(A)U^*+Vg(B)V^*)\|_! \\
&\ge\|f(g(A))\|_!+\|f(g(B))\|_!.
\end{align*}
Letting $\eps\searrow0$ yields the claimed inequality thanks to Proposition
\ref{full-sym-deriv}\,(b). 
\end{proof}

Finally, we return to a general finite von Neumann algebra $\cN$ with a faithful normal
trace $\tau$, $\tau(I)=1$, and extend Theorem \ref{theorem-2} to $\overline{\cN}^+$ with a
restriction on derived anti-norms. For this, we start with a fully symmetric norm $\rho$ on
the commutative von Neumann algebra $L^\infty(0,1)$ with the trace $\int_0^1\cdot\,dt$
(expectation). Define a fully symmetric norm $\|\cdot\|_\rho$ on $\cN$ as
$$
\|X\|_\rho:=\rho(\mu(X)),\qquad X\in\cN,
$$
which we call a {\it $\rho$-symmetric norm}. This way of construction of symmetric norms
is common in the theory of non-commutative Banach function spaces (e.g., \cite{DDP1,DDP2}).
Let $\|\cdot\|_!$ be the (fully symmetric) derived anti-norm on $\cN^+$ that is derived
from $\|\cdot\|_\rho$ and a $p>0$. In case of a diffuse $\cM$, any fully symmetric norm on
$\cM$ is a $\rho$-symmetric norm with
$$
\rho(h):=\bigg\|\int_0^1h(t)\,dF_t\bigg\|,\qquad h\in L^\infty(0,1),
$$
so the restriction here on fully symmetric derived anti-norms is indeed no restriction
in the diffuse case. However, on a general $\cN$ we have a fully symmetric norm which is
not written as a $\rho$-symmetric norm.

Consider the tensor product (diffuse) von Neumann algebra
$$
\cM:=\cN\,\overline{\otimes}\,L^\infty(0,1)
$$
with the tensor product trace $\tau\otimes\int_0^1\cdot\,dt$, and define the $\rho$-symmetric
norm $\|\cdot\|_\rho$ on $\cM$ and the corresponding derived anti-norm $\|\cdot\|_!$ on
$\cM^+$ in the same way as above. Then the following equations are obvious:
$$
\|X\|_\rho=\|X\otimes1\|_\rho,\quad X\in\overline{\cN}\ ;\qquad
\|A\|_!=\|A\otimes1\|_!,\quad A\in\overline{\cN}^+.
$$
Therefore, all the results concerning fully symmetric derived anti-norms on $\cM^+$ in
Section 6 and in this section remain true for the derived anti-norm $\|\cdot\|_!$ on $\cN^+$
from $\|\cdot\|_\rho$ as above. In particular, we have

\begin{cor}\label{cor-super}
Let $\|\cdot\|_!$ be a derived anti-norm on $\cN^+$ that is derived from a $\rho$-symmetric
norm $\|\cdot\|_\rho$ on $\cN$ and a $p>0$. Then, for every $A,B\in\overline{\cN}^+$ and
every function $\psi(t)$ in $\cS$,
$$
\|\psi(A+B)\|_!\ge\|\psi(A)\|_!+\|\psi(B)\|_!.
$$
\end{cor}

\vskip 5pt
\begin{cor} Let $g:[0,\infty)\to[0,\infty)$ be a convex function with $g(0)=0$, and let
$\psi(t)$ be a strictly increasing function in $\cS$. Then, for all nonsingular
$A,B\in\cN^+$ and all $0<p\le 1$,
\begin{equation*}
\frac{\tau(g^p(A+B))}{\tau(\psi^{p-1}(A+B))} \ge
\frac{\tau(g^p(A))}{\tau(\psi^{p-1}(A))} + \frac{\tau(g^p(B))}{\tau(\psi^{p-1}(B))}.
\end{equation*}
\end{cor}

\vskip 5pt
\begin{proof}
By Theorem \ref{theorem-1} the functional $A\mapsto\{\tau(g^p(A))\}^{1/p}$ is superadditive
with finite values on $\cN^+$. By the previous corollary,
$A\mapsto\{\tau(\psi^{p-1}(A))\}^{1/(p-1)}$ is superadditive with strictly positive finite
values on the nonsingular part of $\cN^+$. Hence, their $p$-weighted geometric mean is
again superadditive on the nonsingular part of $\cN^+$.
\end{proof}

Note that $\|\cdot\|_{(t)}$ is a $\rho$-symmetric norm and \eqref{limit-form}
is valid in $\cN^+$. Hence, through Corollary \ref{cor-super} applied to the derived
anti-norm from $\|\cdot\|_{(t)}$, we have
$$
\Delta_t(\psi(A+B))\ge\Delta_t(\psi(A))+\Delta_t(\psi(B)),\qquad A,B\in\overline{\cN}^+
$$
for all $t\in(0,1]$ and all functions $\psi(t)$ in $\cS$. For the Fuglede-Kadison
determinant $\Delta$, since
$\Delta(\sqrt{\psi\omega}(A))=\{\Delta(\psi(A))\Delta(\omega(A))\}^{1/2}$ for
$\psi,\omega:[0,\infty)\to[0,\infty)$ and $A\in\overline{\cN}^+$, we furthermore have

\begin{cor}
Let $\psi(t)$ and $\omega(t)$ be two functions in $\cS$. Then, for all
$A,B\in\overline{\cN}^+$,
$$
\Delta\left(\sqrt{\psi\omega}(A+B)\right)\ge
\Delta\left(\sqrt{\psi\omega}(A)\right)+\Delta\left(\sqrt{\psi\omega}(B)\right).
$$
\end{cor}

This is a substantial generalization of the Minkowski inequality for $\Delta(A)$
on $\cM^+$ given in \cite{Ar} as a consequence of a variational expression of $\Delta$. In
addition, it is worth noting that the concavity of $A\mapsto\Delta(f(A))$ on $\cM^{sa}$ for
a positive concave function $f$ was shown in \cite{LP} (a similar result for matrices is in
\cite{BH1}).

\vskip 20pt

{\small
Jean-Christophe Bourin

Laboratoire de math\'ematiques, Universit\'e de Franche Comt\'e,

25030 Besan\c con, France

jcbourin@univ-fcomte.fr

\vskip 20pt
Fumio Hiai

Tohoku University (Emeritus),

Hakusan 3-8-16-303, Abiko 270-1154, Japan

hiai.fumio@gmail.com

}

\end{document}